\documentclass[a4paper]{elsarticle}

\usepackage[latin1]{inputenc}
\usepackage{amssymb,amsthm,amsmath}    
\usepackage[english]{babel}
\usepackage{thmtools} %
\usepackage{graphicx} 
\usepackage{tikz} 
\usetikzlibrary{patterns,shapes,positioning,matrix}
\usepackage{hyperref} 
\addto\extrasenglish{

}

\theoremstyle{definition}
\newtheorem{definition}{Definition}[section]
\newtheorem{example}[definition]{Example}

\theoremstyle{plain}
\newtheorem{proposition}[definition]{Proposition}
\newtheorem{lemma}[definition]{Lemma}
\newtheorem{theorem}[definition]{Theorem}
\newtheorem{corollary}[definition]{Corollary}
\newtheorem{conjecture}[definition]{Conjecture}
\newtheorem{openproblem}[definition]{Open problem}

\newtheorem{claim}[definition]{Claim}

\theoremstyle{remark}
\newtheorem{remark}[definition]{Remark}

\newcommand{\N}{\mathbb{N}} 

\newcommand{\words}{\Sigma^*} 
\newcommand{\eps}{\varepsilon} 
\newcommand{\pref}{\textstyle{\mathop{\mathrm{pref}}}} 
\newcommand{\suff}{\textstyle{\mathop{\mathrm{suff}}}} 

\begin{document}
\begin{frontmatter}

\title {On cardinalities of $k$-abelian equivalence classes}
\author[Turku]{Juhani Karhum\"aki\fnref{fn1,fn2}}
\ead{karhumak@utu.fi}
\fntext[fn1]{Department of Mathematics and Statistics, University of Turku, Finland}
\fntext[fn2]{TUCS Turku Centre for Computer Science, Turku, Finland}

\author[Lyon]{Svetlana Puzynina\fnref{fn5,fn4}}
\ead{s.puzynina@gmail.com}
\fntext[fn5]{LIP, ENS de Lyon, CNRS, UCBL, Universit\'e de Lyon}
\fntext[fn4]{Sobolev Institute of Mathematics, Russia}

\author[Lyon]{Micha\"el Rao\fnref{fn5}}
\ead{michael.rao@ens-lyon.fr}

\author[Turku]{Markus A. Whiteland\corref{cor1}
\fnref{fn1}}
\ead{mawhit@utu.fi}
\cortext[cor1]{Corresponding author}

\address[Turku]{20014 University of Turku, Finland}
\address[Lyon]{LIP, ENS Lyon, 46 All\'ee d'Italie, Lyon 69364, France}

\begin{abstract}Two words $u$ and $v$ are $k$-abelian equivalent
if for each word $x$ of length at most $k$, $x$ occurs equally
many times as a factor in both $u$ and $v$. The notion of $k$-abelian
equivalence is an intermediate notion between the abelian equivalence
and the equality of words. In this paper, we study the equivalence
classes induced by the $k$-abelian equivalence, mainly focusing on
the cardinalities of the classes. In particular, we are interested
in the number of singleton $k$-abelian classes, i.e., classes
containing only one element. We find a connection between the
singleton classes and cycle decompositions of the de Bruijn graph.
We show that the number of classes of words of length $n$
containing one single element is of order
$\mathcal O (n^{N_m(k-1)-1})$, where
$N_m(l)=\tfrac{1}{l}\sum_{d\mid l}\varphi(d)m^{l/d}$ is the number
of necklaces of length $l$ over an $m$-ary alphabet.
We conjecture that the upper bound is sharp. We also remark that,
for $k$ even and $m=2$, the lower bound $\Omega (n^{N_m(k-1)-1})$
follows from an old conjecture on the existence of Gray codes
for necklaces of odd length. We verify this conjecture for necklaces
of length up to 15.
\end{abstract}
\end{frontmatter}

\section{Introduction}
For an integer $k$, two finite words $u,v$ are
called \emph{$k$-abelian equivalent}, denoted by $u\sim_k v$, if
they contain the same number of occurrences of each non-empty word
of length at most $k$. The notion has captured attention recently,
especially with respect to repetitions and complexity functions
of infinite words (\cite{EMMN15,HSa13,KPS12,KSZ13,KSZ14}). In
particular, the work of J. Karhum\"aki, A. Saarela and L. Q. Zamboni
(\cite{KSZ13}) includes several equivalent characterizations of
$k$-abelian equivalence and sets a solid foundation for the
investigation of the topic.

This paper can be seen as a step towards understanding the structure
of $k$-abelian equivalence classes. First, we characterize $k$-abelian
equivalence in terms of rewriting. More precisely, we introduce the
notion of \emph{$k$-switchings}, where one rearranges factors
occurring in the word (see \autoref{sec:characterization} for the
definition.) Using this characterization, we are able to completely
characterize $k$-abelian classes containing only one word. Furthermore,
we show that the number of such classes is of order
$\mathcal O (n^{N_m(k-1)-1})$, where
$N_m(l)=\tfrac{1}{l}\sum_{d\mid l}\varphi(d)m^{l/d}$ is the number
of necklaces (also known as circular words; i.e., equivalence
classes of words under conjugacy) of length $l$ over an
$m$-ary alphabet and $\varphi$ is Euler's totient function.
We also obtain a formula for counting the number of words in a
$k$-abelian equivalence class induced by a given word $w$.

A common theme of the results mentioned above is that they are obtained
by forming a connection between properties of the de Bruijn graph and
$k$-abelian equivalence classes. We invoke classical theorems from graph
theory (e.g., BEST theorem \cite{A-EdB51}) and also particular results
concerning de Bruijn graphs (such as Lempel's conjecture proved by
Mykkeltveit \cite{Myk72}).
We observe a connection between cycle decompositions of the de Bruijn
graph and $k$-abelian singleton classes, and we use this connection to
find an upper bound of the number of $k$-abelian singleton classes of
a given length, which we conjecture to be sharp (up to a constant
multiple).

For the binary alphabet and even $k$, the lower bound follows from a
twenty-year-old open problem on existence of Gray codes for necklaces
stated in \cite{Sav97}, (see \autoref{con:GrayCode}, see also
\cite{DeDr07} for recent results). A Gray code for necklaces is defined
as an ordering of all necklaces such that any two consecutive necklaces
have representatives which differ in only one bit. Concerning the
conjecture, we give new supporting evidence.

The paper is structured as follows. In \autoref{sec:preliminaries}
we define the basic notions and recall basic results on
combinatorics on words. In \autoref{sec:characterization} we show
a new characterization of $k$-abelian equivalence in terms of
switchings. In \autoref{sec:cardinality} we obtain a formula for
counting the number of words in an equivalence class induced by a
given word $w$. In \autoref{sec:singletons} we study, based on our
characterization, the number of singleton classes, i.e., $k$-abelian
classes containing exactly one element. We give an upper bound for the
number of singleton classes, and we conjecture it to be sharp. We then
finish by stating an open problem in \autoref{sec:conclusions}.

\section{Preliminaries}\label{sec:preliminaries}
Given a finite non-empty set $\Sigma$, we denote by $\words$ the set of finite
words over $\Sigma$ including the empty word $\eps$. The set of non-empty words
is denoted by $\Sigma^+$. Given a finite word
$u = a_1a_2\cdots a_n$, $a_i\in \Sigma$, $n\geq 1$, we let $|u|$ denote the
length $n$ of $u$ and, by convention, we set $|\eps| = 0$. We denote by
$\Sigma^n$ the words of length $n$ over $\Sigma$. For any $x\in \Sigma^+$ let
$|u|_x$ denote the number of occurrences of $x$, including overlapping ones, as
a factor of $u$. We shall use the convention $|u|_{\eps} = |u|+1$.

For $u = a_1a_2\cdots a_n$ and integers $i,j$ with $1\leq i\leq j \leq n$ we
will use the notation $u{[i,j]} = a_i\cdots a_j$, for $j>i$ we denote
$u{[i,j)} = a_i\cdots a_{j-1}$. For $i\geq j$, we define $u{[i,j)} = \eps$. We
shall often denote $u[i..] = u[i,|u|]$ for brevity. 

Let $w$ be a non-empty word, and $q\in \mathbb{Q}$, such that $q \cdot |w| \in \N$.
Then $w^q$ is defined as the word of length $q \cdot |w|$ for which $w^q[i]=w[i]$
if $i\le \min\{q\cdot |w|,|w|\}$ and $w^q[i]=w^q[i-|w|]$ otherwise. For $q\geq 2$,
we call the word $w^q$ a \emph{repetition}. A word is \emph{primitive} if there is
no word $v$ and integer $l>1$ such that $w=v^l$. The \emph{period} of a word is the
least integer $l$ such that for every $1\le i \le |w|-l$, $w[i]=w[i+l]$. Note that
if $w$ is primitive, then the period of $w^q$ is $|w|$.

Two words $u$ and $v$ are \emph{conjugates}, if $u = xy$ and $v = yx$
for some $x,y\in\words$. The set of all conjugates of a word $u$ is
called a \emph{necklace}, or a \emph{circular word}, induced
by $u$. The necklace induced by $u$ is denoted by $u^{\circ}$.

For a finite or infinite word $u$ we denote by $F(u)$ the set of
finite factors of $u$ and by $F_n(u)$ the set of factors of $u$ of
length $n$. Similarly, for a non-empty necklace $u^{\circ}$, we
define the set of factors of $u^{\circ}$ as
$F(u^{\circ}) = F(u^{\omega})$ and factors of length $n$ as
$F_n(u^{\circ}) = F_n(u^{\omega})$,
where $u^{\omega}$ is the infinite repetition $uuu\cdots.$

The \emph{de Bruijn graph of order} $n$ over $\Sigma$, denoted by
$dB_{\Sigma}(n)$, is defined as follows. The set of vertices equals
$\Sigma^n$. There is an edge from $u$ to $v$ if and only if there
exist $a,b \in \Sigma$ and a word $x\in\Sigma^{n-1}$ such that
$u = ax$ and $v = xb$. The edge $(ax,xb)$ corresponds to the word
$axb$ of length $n+1$. We shall often omit the subscript $\Sigma$
when the alphabet is clear from context.

Define the function $\Psi_k : \words \to \N^{\Sigma^k}$ as
$\Psi_k(u)[x] = |u|_x$ for $x \in \Sigma^k$. For $k=1$, $\Psi_k(u)$
is also known as the \emph{Parikh vector} of a word $u\in\words$.

\begin{definition}
Two words $u,v\in\words$ are said to be $k$-abelian equivalent,
denoted by $u\sim_k v$, if $\Psi_m(u) = \Psi_m(v)$ for all
$1\leq m \leq k$.
\end{definition}
Note that $u\sim_k v$ implies, by definition, $u\sim_m v$ for all
$1\leq m\leq k$. The relation $\sim_k$ is clearly an equivalence
relation, in fact, even a congruence. We shall denote by $[u]_k$
the $k$-abelian equivalence class induced by $u$.

\begin{definition}
Let $u\in\words$ and $k\geq 1$. If $|[u]_k| = 1$, then $u$ is said
to be a \emph{$k$-abelian singleton}, or in short \emph{singleton},
when $k$ is clear from context.
\end{definition}

\begin{example}
Let $u = ababab$ and $v = aababb$. Then $u$ is a $2$-abelian singleton,
since $[u]_2 = \{u\}$ as there are no other words of length 6 containing
three occurrences of $ab$. On the other hand, $[v]_2 = \{v,aabbab,abaabb,abbaab\}$.
\end{example}

The following characterization is easy to see, see e.g. \cite{KSZ13}:
\begin{lemma}
Let $u$ and $v$ be words of length at least $k-1$. Then $u\sim_k v$
if and only if $\Psi_k(u) = \Psi_k(v)$, $\pref_{k-1}(u) = \pref_{k-1}(v)$
and $\suff_{k-1}(u) = \suff_{k-1}(v)$.
\end{lemma}
We shall mostly use this equivalent definition for $k$-abelian equivalence,
that is, we generally assume that the words are long enough.

\section{A characterization by rewriting}\label{sec:characterization}
In this section we describe rewriting rules of words, which preserve
equivalence classes. This provides a new characterization of $k$-abelian
equivalence.

Let $k \geq 1$ and let $u = u_1\cdots u_n$. Suppose further that there exist
indices $i,j,l$ and $m$, with $i<j \leq l<m\leq n-k+2$, such that
$u{[i,i+k-1)} = u{[l,l+k-1)} = x\in\Sigma^{k-1}$ and
$u{[j,j+k-1)} = u{[m,m+k-1)} = y\in\Sigma^{k-1}$. We thus have
$$u = u{[1,i)} \cdot u{[i,j)} \cdot u{[j,l)} \cdot u{[l,m)} \cdot u{[m..]},$$
where $u{[i..]}$ and $u{[l..]}$ begin with $x$ and $u{[j..]}$ and $u{[m..]}$
begin with $y$. Note here that we allow $l = j$ (in this case $y = x$). We
define a \emph{$k$-switching} on $u$, denoted by $S_{u,k}(i,j,l,m)$, as
\begin{equation}\label{eq:k-switching}
    S_{u,k}(i,j,l,m) = u[1,i) \cdot u[l,m) \cdot u[j,l) \cdot u[i,j) \cdot u[m..].
\end{equation}

Roughly speaking, the idea is to switch the positions of two factors who both
begin and end with the same factors of length $k-1$, and we allow the situation
where the factors can all overlap. We remark that, in the case of $j=l$,
$k$-switchings were considered in a different context in \cite{CdL04}.

\begin{example}
    Let $u = aabababaaabab$ and $k=4$. Let then $x = aba$, $y = bab$, $i = 2$,
		$j=3$, $l = 4$ and $m = 11$. We then have
		\begin{align*}
										u	&= a \cdot a \cdot b \cdot ababaaa\cdot bab\\
    S_{u,4}(i,j,l,m)	&= a \cdot ababaaa\cdot b \cdot a \cdot bab.
		\end{align*} One can check that $u \sim_4 S_{u,4}(i,j,l,m)$. Note that in this
		example the occurrences of $x$ and $y$ are overlapping. 
\end{example}

In other words, for a word $u = a_1\cdots a_n$, a $k$-switching
$S_{u,k}(i,j,l,m) = v$ can be seen as a permutation $\sigma$ on the
set $\{1,\ldots,n\}$:
\begin{equation*}
\sigma: (1, \ldots, n) \mapsto (1, 2, \ldots, i-1, l, \ldots, m-1,
j,  \ldots, l-1, i,  \ldots, j-1, m, \ldots, n),
\end{equation*}
so that $v = a_{\sigma(1)}\cdots a_{\sigma(n)}$.
We remark that this permutation can also be considered
as a discrete interval exchange transformation.

We now show that performing a $k$-switching on a word
does not affect the number of occurrences of factors
of length $k$.
\begin{lemma}\label{lem:switchingImpliesEquivalence}
Let $u \in\words$ and $v = S_{u,k}(i,j,l,m)$ be a
$k$-switching on $u$. Then $u \sim_k v$.
\end{lemma}
\begin{proof}
\begin{figure}
    \centering
        \begin{tikzpicture}
            \node at (-5.5,0) {\footnotesize{$u$}};
								\draw[fill=gray] (-5,0) rectangle (-4,.1);
                \draw[fill=gray] (3,0) rectangle (5,.1);
                \draw[pattern=north east lines] (-2,0) rectangle (.5,.1);
                \draw[pattern=crosshatch] (-4,0) rectangle (-2.5,.2);
								\draw[fill=white] (-4,0) rectangle (-3.5,.1);
                    \node at (-3.9,-.2) {\tiny{$i$}};
                \draw[fill] (-3,0) circle (.03);
                    \node at (-3,-.2) {\tiny{$p$}};

                \draw[fill] (-2.5,0) rectangle (-2,.1);
                    \node at (-2.4,-.2) {\tiny{$j$}};

                \draw[fill] (2.5,0) rectangle (3,.1);
                \node at (2.6,-.2) {\tiny{$m$}};

                \draw[pattern=crosshatch dots] (.5,0) rectangle (2.5,.2);
								\draw[fill=white] (.5,0) rectangle (1,.1);
                \node at (.6,-.2) {\tiny{$l$}};

            \node at (-5.5,-1.5) {\footnotesize{$v$}};
								\draw[fill=gray] (-5,-1.5) rectangle (-4,-1.4);
                \draw[fill=gray] (3,-1.5) rectangle (5,-1.4);
                \draw[pattern=north east lines] (-1.5,-1.5) rectangle (1,-1.4);
                \draw[pattern=crosshatch dots] (-4,-1.5) rectangle (-2,-1.3);
								\draw[fill=white] (-4,-1.5) rectangle (-3.5,-1.4);
                    \node at (-3.9,-1.7) {\tiny{$\sigma(l)$}};
                \draw[fill] (-2,-1.5) rectangle (-1.5,-1.4);
                    \node at (-1.9,-1.7) {\tiny{$\sigma(j)$}};
                \draw[pattern=crosshatch] (1,-1.5) rectangle (2.5,-1.3);
								\draw[fill=white] (1,-1.5) rectangle (1.5,-1.4);
                    \node at (1.1,-1.7) {\tiny{$\sigma(i)$}};
                    \draw[fill] (2,-1.5) circle (.03);
                    \node at (2,-1.7) {\tiny{$\sigma(p)$}};
                \draw[fill] (2.5,-1.5) rectangle (3,-1.4);
                    \node at (2.6,-1.7) {\tiny{$m$}};
        \end{tikzpicture}
        \caption{Illustration of a $k$-switching. Here the white rectangles symbolize
				$x$ and the black rectangles symbolize $y$.}
        \label{fig:k-switching}
\end{figure}
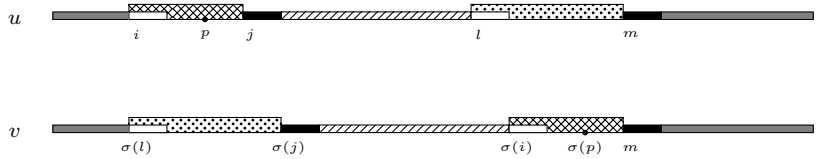
Let $\sigma$ be the permutation corresponding to the $k$-switching as
described above. It is straightforward to verify that, for any $p$,
$p\leq n-k+1$, the factor of length $k$ beginning at index $p$ in $u$
is equal to the factor starting at index $\sigma(p)$ in $v$ (see
\autoref{fig:k-switching}, where we have the case of no overlaps of the
factors $x$ and $y$. The other cases are analogous). This implies that
$\Psi_k(u) = \Psi_k(v)$. Furthermore, $\pref_{k-1}(u) = \pref_{k-1}(v)$
and $\suff_{k-1}(u) = \suff_{k-1}(v)$. It follows that $u\sim_k v$.
\end{proof}

Let us define a relation $R_k$ of $\words$ with $uR_kv$ if and only if
$v = S_{u,k}$ for some $k$-switching on $u$. Now $R_k$ is clearly symmetric,
so that the reflexive and transitive closure $R_k^*$ of $R_k$ is an
equivalence relation.

In this terminology, the above lemma asserts that $uR_k^*v$ implies
$u\sim_k v$. We now prove the converse, so that the relations $\sim_k$
and $R_k^*$ actually coincide.

\begin{proposition}\label{prop:rewritingCharacterization}
For two words $u,v\in\words$, we have $u\sim_k v$ if and only if
$u R_k^* v$.
\end{proposition}

For the proof of the proposition, we need the following technical claim
which will also be used later:
\begin{claim}\label{claim:technicalSwitching}
	Let $w\sim_k w'$, $w\neq w'$. Let $\lambda x$ be the longest common prefix of
	$w$ and $w'$ with $\lambda \in \words$, $x\in \Sigma^{k-1}$, whence
	$w = \lambda x a\mu$ and $w' = \lambda x b\mu'$ for some $\mu, \mu'\in \words$,
	$a,b\in\Sigma$, $a\neq b$. Then there exist $y\in\Sigma^{k-1}$ and indices
	$j,l,m$, with $|\lambda|+1 < j\leq l < m$, such that
	\begin{align*}
		w[j,j+k-1) &= y, & w[l,l+k) &= xb,\text{ and} & w[m,m+k-1) &= y.
	\end{align*}
\end{claim}
\begin{proof}
It follows from $w\sim_k w'$ that $w'$ has an occurrence of $xa$ and $w$
has an occurrence of $xb$ occurring after the common prefix $\lambda$.
We let $i = |\lambda|+1$ be the position (i.e., the starting index
of the occurrence) of $xa$ in $w$ and let $l$ be the minimal position
(leftmost occurrence) of $xb$ in $w$ with $l>i$. Let $p$ be a position
of $xa$ in $w'$ with $p>i$, (see \autoref{fig:illustration_of_contradiction}).

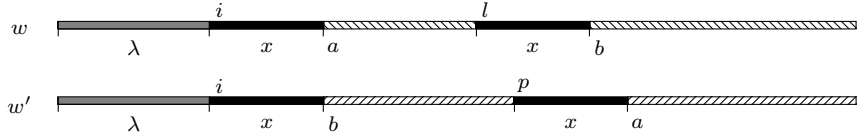
\begin{figure}
    \centering
\begin{tikzpicture}
\node at (-7,0) {\footnotesize{$w$}};
    \draw [fill=gray] (-6.5,0) rectangle (-4.5,.1);
		\draw [thin,|-|] (-6.5,0) -- (-4.5,0);
    \node at (-5.5,-.25) {\footnotesize{$\lambda$}};

    \draw [fill=black] (-4.5,0) rectangle (-3,.1);
		\draw [thin,-|] (-4.5,0) -- (-3,0);
    \node at (-3.75,-.25) {\footnotesize{$x$}};
    \node at (-4.375,.25) {\footnotesize{$i$}};
    \node at (-2.875,-.25) {\footnotesize{$a$}};
		
		\draw [pattern=north west lines] (-3,0) rectangle (-1,.1);

    \draw [thin,|-|] (-1,0)--(.5,0);
		\draw [fill=black] (-1,0) rectangle (.5,.1);
    \node at (-.25,-.25) {\footnotesize{$x$}};
    \node at (-.875,.25) {\footnotesize{$l$}};
    \node at (.625,-.25) {\footnotesize{$b$}};
		
		\draw [pattern=north west lines] (.5,0) rectangle (4,.1);

\node at (-7,-1) {\footnotesize{$w'$}};
		\draw [fill = gray] (-6.5,-1) rectangle (-4.5,-.9);
    \draw [thin,|-|] (-6.5,-1)--(-4.5,-1);
    \node at (-5.5,-1.25) {\footnotesize{$\lambda$}};

		\draw [fill=black] (-4.5,-1) rectangle (-3,-.9);
    \draw [thin,-|] (-4.5,-1)--(-3,-1);
    \node at (-3.75,-1.25) {\footnotesize{$x$}};
		\node at (-4.375,-.75) {\footnotesize{$i$}};

		\draw [pattern=north east lines] (-3,-1) rectangle (-.5,-.9);
    \node at (-2.875,-1.25) {\footnotesize{$b$}};

		\draw [fill=black] (-.5,-1) rectangle (1,-.9);
    \draw [thin,|-|] (-.5,-1) -- (1,-1);
    \node at (.25,-1.25) {\footnotesize{$x$}};
    \node at (1.125,-1.25) {\footnotesize{$a$}};
		\node at (-.375,-.75) {\footnotesize{$p$}};

		\draw [pattern = north east lines] (1,-1) rectangle (4,-.9);
\end{tikzpicture}
\caption{Illustration of the proof of \autoref{claim:technicalSwitching}.}
\label{fig:illustration_of_contradiction}
\end{figure}

Consider then the set $F_k(w'[i..])$; each word in this set occurs
somewhere in $w[i..]$, since $w\sim_k w'$. Let then $q$, $q\geq i$,
be the minimal index such that the factor $w'[q,q+k)$ occurs in
$w[i,l+k-1)$.
Such an index exists since, for example, $w'[p,p+k) = w[i,i+k)$.
Moreover, by the minimality of $l$, we have $q>i$. Let $y = w'[q,q+k-1)$
and let $j'$, $i\leq j'\leq l-1$, be a position of $y$ in $w$. We shall
now choose the index $j$ in the claim. If $j' > i$ we choose $j = j'$.
If $j' = i$, then necessarily $x = y$ and we choose $j = l$.

We shall now choose the index $m$ in the claim. By the choice of $q$,
we have that $w'[q-1,q+k-1)$, an element of $F_k(w'[i..])$, occurs at
some position $m'$, $m'\geq l$, in $w$. It follows that $y$ occurs in $w$
at position $m = m'+1$, with $m > l$. We have now obtained the factor $y$
and the positions of $y$ and $xb$ as claimed.
\end{proof}	

\begin{proof}[Proof of \autoref{prop:rewritingCharacterization}]
It is enough to show that $u\sim_k v$ implies $u R_k^* v$, since
the converse follows from \autoref{lem:switchingImpliesEquivalence}.

More precisely, we shall prove the following: Let $u\sim_k v$ and suppose
that for the longest common prefix $\nu$ of $u$ and $v$, we have
$|\nu|<|v|$. Then there exists a word $z$ such that $u R_k z$ and the
longest common prefix of $z$ and $v$ has length at least $|\nu| + 1$. It
is clear that \autoref{prop:rewritingCharacterization} follows
immediately from \autoref{lem:switchingImpliesEquivalence} and this observation.

Indeed, applying \autoref{claim:technicalSwitching} to $w = u$ and
$w' = v$, with $\nu = \lambda x$, we obtain indices $i,j,l,m$ which
give rise to a $k$-switching $S_{u,k}(i,j,l,m) = z$, such that the
longest common prefix of $z$ and $v$ has length at least $|\nu|+1$.
This concludes the proof.
\end{proof}

\section{The cardinality of a \texorpdfstring{$k$}{k}-abelian equivalence class}\label{sec:cardinality}
In this section we analyze the sizes of $k$-abelian equivalence
classes. One of the interesting questions there is the following:
Given $n$ and $\Sigma$, which cardinalities of $k$-abelian classes
of words of length $n$ over the alphabet $\Sigma$ exist? We begin
with a simple observation:

\begin{claim}\label{claim:example}1) For any pair $n,m \in \N$, with
$m\leq n-2$, there exists a ternary word $w$, such that $|w|=n$ and $|[w]_2| = m$.

\noindent 2) For any $p\in\N$ there exists a sequence of binary words
$(w_n)_{n\in\N}$, $|w_n| = n$, such that $|[w_n]_2| = \Theta(n^p)$.
\end{claim}

\begin{proof} 1) Choose $w = a^{n-m-2}cbc^{m}$. The claim
follows since any $w'\sim_2 w$ has $a^{n-m-2}$ as a prefix and
$|[cbc^m]_2| = m$ for all $m\in\N$.

\noindent 2) Let $w_n = (ab)^pa^{n-2p}$ for $n\geq 2p$. Then
$|[w_n]_2| = \binom{n-p-1}{p} = \Theta(n^p)$. \end{proof}

In general, it is not clear which orders of growth one can achieve.
In the following we will prove a formula for counting $|[w]_k|$
for a given word $w$.

\subsection{Equivalence classes as Eulerian cycles in weighted de Bruijn graphs}
In the following, when talking about graphs, we mean directed
multigraphs with loops. For a fixed graph $G$, we denote by $d_G^+(u)$
(resp., $d_G^-(u)$) the number of outgoing (resp., incoming) edges of $u$.
For $u,v\in V$, the number of edges from $u$ to $v$ in $G$ is denoted by
$m_G(u,v)$. When clear from context, we omit the subscript $G$.

Let $G = (V,E)$ and let the set of vertices be ordered as
$V = \{v_1,\ldots,v_{|V|}\}$. The \emph{adjacency matrix of $G$} is the
matrix $(a_{ij})_{i,j}$, where $a_{ij} = m(v_i,v_j)$.

We repeat an observation made in \cite{KSZ13} connecting $k$-abelian
equivalence with Eulerian paths in certain multigraphs. Let
$f \in \N^{\Sigma^k}$ be an arbitrary vector. We modify the de Bruijn graph
$dB(k-1)$ with respect to $f$ into $G_f = (V,E)$ as follows.
We define $V$ as the set of words $x\in\Sigma^{k-1}$ such that $x$ is a
prefix or a suffix of a word $z\in\Sigma^k$ for which $f[z]>0$.
We define the set of edges as follows: for each $z\in\Sigma^k$ with $f[z]>0$, we take the edge from $u$ to $v$ with multiplicity $f[z]$,
where $u$ is the length $k-1$ prefix of $z$, and $v$ is the length $k-1$ suffix of $z$.

Note that for $f = \Psi_k(w)$, the graph $G_f$ resembles the
\emph{Rauzy graph} of $w$ of order $k-1$ (see \cite{Rau82}), with
$V = F_{k-1}(w)$ and the edges of $G_f$ correspond to the set
$F_k(w)$ with multiplicities.

In the following,  for $u,v \in \Sigma^{k-1}$, we denote by $\Sigma(u,v)$
the set of words which begin with $u$ and end with $v$: $\Sigma(u,v) = u \words \cap \words v$.

\begin{lemma}[Lemma 2.12 in \cite{KSZ13}]\label{lem:graphStructure}
For a vector $f \in \N^{\Sigma^k}$ and words $u,v \in \Sigma^{k-1}$, the following
are equivalent:
\begin{enumerate}
  \item there exists a word $w\in \Sigma(u,v)$ such that $f = \Psi_k(w)$,
  \item $G_f$ has an Eulerian path starting from $u$ and ending at $v$,
  \item the underlying graph of $G_f$ is connected, and $d^-(s)=d^+(s)$
		for every vertex $s$, except that if $u\neq v$, then $d^-(u) = d^+(u)-1$
		and $d^-(v)=d^+(v)+1$.
\end{enumerate}
\end{lemma}

The following corollary is immediate.
\begin{corollary}\label{cor:EulerianCycleEquivalence}
For a word $w\in \Sigma(u,v)$ and $k\geq 1$, we have that $w' \sim_k w$ if and only if
$w'$ induces an Eulerian path from $u$ to $v$ in $G_{\Psi_k(w)}$.
\end{corollary}

\subsection{On Eulerian cycles in directed multigraphs}
We recall some notions and well-known results from graph theory.
\begin{definition}
Let $G=(V,E)$ be a graph. The \emph{Laplacian matrix} $\Delta$ of
$G$ is defined as
\[\Delta_{uv} = \begin{cases} -m(u,v), & \mbox{ if } u\neq v,\\
  d^+(u)-m(u,v), & \mbox{ if } u=v. \end{cases}\]
\end{definition}
For the Laplacian $\Delta$ of a graph $G$ and a vertex $v$ of $G$,
we denote by $\Delta(v)$ the matrix obtained by removing from $\Delta$
the row and column corresponding to $v$.

\begin{remark}
	We note that for a directed multigraph $G$ and a vertex $v$,
	$\det(\Delta(v))$ counts the number of \emph{rooted spanning trees with
	root $v$} in $G$. This result is known as Kirchhoff's matrix tree theorem
	(for a proof, see \cite{A-EdB51}).
\end{remark}

A graph $G$ is called \emph{Eulerian} if there exists an Eulerian
\emph{cycle}. We recall the BEST theorem, first discovered by C. A. B. Smith
and W. T. Tutte in 1941 and later generalized by T. van Aardenne-Ehrenfest
and N. G. de Bruijn (see \cite{A-EdB51}). For this, let $\epsilon(G)$ denote
the number of distinct Eulerian cycles in an Eulerian graph $G$. Here two
cycles are considered to be the same, if one is a cyclic shift of the other.
Equivalently, $\epsilon(G)$ counts the number of distinct Eulerian cycles
beginning from a fixed edge $e$.
\begin{theorem}[BEST theorem]\label{thm:BEST}
Let $G$ be a connected directed Eulerian multigraph. Then
\begin{equation*}
\epsilon(G) = \det(\Delta(u)) \prod_{v\in V}(d^+(v)-1)!,
\end{equation*}
where $\Delta$ is the Laplacian of $G$ and $u$ is any vertex of $G$.
\end{theorem}

\subsection{The cardinality of an equivalence class}
We are now going to count $|[w]_k|$ for any word $w\in\words$ with
$w$ long enough. Let $w\in \Sigma(u,v)$, $u,v\in\Sigma^{k-1}$, and denote
by $f=\Psi_k(w)$. By \autoref{cor:EulerianCycleEquivalence}, we are
interested in the number of Eulerian paths of $G_f$. The difference from
$\epsilon(G_f)$ (in the case $u = v$) is that we consider two cycles to be
distinct if the \emph{vertices} are traversed in a different order.
Nonetheless, we can still use the BEST theorem to obtain this number.
Note that in $G_f$ we have $d^+(x) = |w|_x$ for all $x\neq v$ and
$d^+(v) = |w|_v-1$.

\begin{proposition}\label{prop:classSize}
Let $k\geq 1$ and $w\in \Sigma(u,v)$ for some $u,v\in\Sigma^{k-1}$. Then
\begin{equation*}
	|[w]_k| = \det(\Delta(v)) \prod_{x\in F_{k-1}(w)} \frac{(|w|_x-1)!}{\prod_{a\in\Sigma}|w|_{xa}!},
\end{equation*}
where $\Delta$ is the Laplacian of $G_{\Psi_k(w)}$.
\end{proposition}
\begin{proof}
Let $f = \Psi_k(w)$ and $V = F_{k-1}(w)$. Suppose first that $u = v$,
so that $G_f$ contains an Eulerian cycle. We shall first count the
number of distinct Eulerian cycles starting from \emph{vertex} $v$.
Note here that two cycles are considered distinct if the \emph{edges}
are traversed in a different order.

It follows from the BEST theorem, that the number of Eulerian cycles
starting from vertex $v$ equals
\begin{equation}\label{eq:cyclesFormula}
	d^+(v) \det(\Delta(v)) \prod_{x\in V}(d^+(x)-1)! = \det(\Delta(v)) \prod_{x\in V}(|w|_x-1)!.
\end{equation}
Now two Eulerian cycles are induced by the same word $z$ if
and only if the \emph{vertices} are traversed in the same
order. The claim follows by dividing the right hand side of
equation \eqref{eq:cyclesFormula} by the number of different
ways to order the individual edges between two vertices $x$
and $y$ for all $x,y\in V$:
\begin{equation*}
	\prod_{(x,y)\in E}m(x,y)! = \prod_{x\in V}\prod_{a\in\Sigma} f[xa]!.
\end{equation*}

Suppose then that $u\neq v$. We shall now add to $G_f$ a new
edge $e = (v,u)$ to obtain $H$, an Eulerian graph. Observe that
$d_{H}^+(v) = d_{G_f}^+(v) + 1 = |w_v|$, the rest of the
out-degrees remain the same. Furthermore, the number of Eulerian
paths from $u$ to $v$ in $G_f$ equals the number of Eulerian
cycles beginning with $e$ in $H$. We again invoke the BEST
theorem: the number of Eulerian cycles beginning from edge $e$ is
\begin{align}
\nonumber \det(\Delta(v)) \prod_{x\in V}(d_{H}^+(x)-1)!
		&= \det(\Delta(v))\ d_{G_f}^+(v)!\prod_{\substack{x\in V\\x\neq v}}(d_{G_f}^+(x)-1)!\\
        &= \det(\Delta(v)) \prod_{x\in V}(|w|_x-1)!,\label{eq:cyclesFormula2}
\end{align}
where $\Delta$ can be chosen to be the Laplacian of either $H$ or $G_f$,
since the Laplacians of $G_f$ and $H$ differ only in the row and column
corresponding to $v$.

Similar to the previous case, we are not interested in which order
the edges from $x$ to $y$ are traversed, with one exception: we
have fixed the starting edge $e$. The right hand side of equation
\eqref{eq:cyclesFormula2} should thus be divided by
\begin{equation*}
(m_{H}(v,u)-1)!\prod_{\substack{(x,y)\in E\\(x,y)\neq (v,u)}}m_{H}(x,y)!
        = \prod_{(x,y) \in E}m_{G_f}(x,y)!
            = \prod_{x\in V}\prod_{a\in\Sigma}f[xa]!.
\end{equation*}
The claim follows.
\end{proof}

\begin{example}Let $w = ababaaaa$ and $f = \Psi_2(w)$. We have
\begin{equation*}f = (|w|_{aa},|w|_{ab},|w|_{ba},|w|_{bb}) = (3,2,2,0).\end{equation*}
The Laplacian of $G_f$ is
$\left(\begin{smallmatrix}
 2 & -2 \\ -2 & 2
\end{smallmatrix}\right)$, from which we obtain $\det(\Delta(a)) = 2$. The above proposition then gives us:
\begin{equation*}|[w]_2| = \det(\Delta(a))\cdot \frac{(|w|_a-1)!(|w|_b-1)!}{|w|_{aa}!|w|_{ab}!|w|_{ba}!}
        = 2\cdot \frac{5!\cdot 1!}{3!\cdot 2!\cdot 2!}
				= \binom{5}{2}.\end{equation*}
One should compare this to the proof of \autoref{claim:example}, 2).
\end{example}

\section{On the structure of singleton classes}\label{sec:singletons}
In this section we are interested in the structure of $k$-abelian singleton
classes, i.e., $k$-abelian classes containing exactly one element. There
always exist $k$-abelian singletons for each length $n$, consider for
example $a^n$.

\begin{example}
It is not difficult to verify that the set of $2$-abelian singletons over
$\{a,b\}$ beginning with $a$ is
$a^+b^* \bigcup ab^*a \bigcup (ab)^*\{\eps, a\}$. As the number of
singleton classes beginning with $b$ are the same up to switching $a$'s
with $b$'s, the total number of $2$-abelian singleton classes of length
$n$ over a binary alphabet is $2n+4$ for $n\geq 4$.
\end{example}

\subsection{A factorization of \texorpdfstring{$k$}{k}-abelian singletons}
We first characterize $k$-abelian singletons in terms of generalized
return words using $k$-switchings. For this we say that $x$ is a
\emph{proper factor} of $w$ if $x$ occurs in $w[2,|w|)$.

\begin{definition}
Let $u\in\words$ and let $x,y\in\Sigma^+$ be of the same length.
A \emph{return from $x$ to $y$ in $u$} is a word $v\in \Sigma^+$
such that $vy$ is a factor of $u$, $x$ is a prefix of $vy$ and
neither $x$ or $y$ occurs as a proper factor of $vy$.
If $x=y$ then we simply say $v$ is a return to $x$ in $w$.
\end{definition}
Note that if $v$ and $v'$, $|v|\leq |v'|$, are distinct returns
from $x$ to $y$ in a word $w$, then $vy$ cannot be a factor of
$v'y$, as otherwise $v'y$ would contain either $x$ or $y$ as a
proper factor. On the other hand, $v$ could be a proper prefix
of $v'$, consider for example $x = a$, $y=b$, $vy = acb$, $v'y = accb$.
\begin{proposition}\label{prop:characterizationReturnwords}
	A word $w\in\words$ is a $k$-abelian singleton if and only if
	for each pair $x,y\in F_{k-1}(w)$ there is at most one return
	from $x$ to $y$ in $w$.
\end{proposition}

\begin{proof}
	We first prove the "only if" part.
	Suppose that $w$ contains two distinct returns $v$ and $v'$
	from $x$ to $y$, $x,y\in\Sigma^{k-1}$. We will show that $w$ is
	not a $k$-abelian singleton.
	
	Let $vy = w[i,j)y$ and $v'y = w[l,m)y$ with $i<l$. Note that $j < m$
	as otherwise $vy$ contains $v'y$ as a factor.
	In fact, by definition, we necessarily have $i<j$ and $l<m$
	(since $v,v'\in\Sigma^+$) and $j\leq l$, (since otherwise $vy$
	contains $x$ as a proper factor).
	Now we can perform a switching $w' = S_{w,k}(i,j,l,m)$ so that
	\begin{equation*}w' = w[1,i)w[l,m)w[j,l)w[i,j)w[m..]\end{equation*}
	Note that $w'\neq w$, since $w$ begins with $w[1,i)vy$ and $w'$ with
	$w[1,i)v'y$. We conclude that $w$ is not a $k$-abelian singleton.
	
	Now we prove the "if" part.
	Suppose that for each pair $x,y\in F_{k-1}(w)$ there is
	at most one return from $x$ to $y$ in $w$. We claim that $w$ is
	a $k$-abelian singleton. Suppose, for the sake of contradiction,
	that $w'\sim_k w$ with $w'\neq w$.
	By applying \autoref{claim:technicalSwitching} to $w$ and $w'$
	we obtain $x,y\in\Sigma^{k-1}$, $a,b\in\Sigma$, $a\neq b$, and
	indices $i,j,l,m$, $i<j\leq l < m$, such that $w[i,i+k) = xa$,
	$w[l,l+k) = xb$, and $w[j,j+k-1) = w[m,m+k-1) = y$.
	At this point we can assume that $j$ and $m$ are minimal among such
	indices.
	
	We first observe that there exists a return to $x$ in $w$ which begins
	with $xa$. If now $x$ occurred in $w[l+1..]$, we would have another
	return to $x$ in $w$ which begins with $xb$, a contradiction. It follows
	that $w[l,m)y$ is a return from $x$ to $y$ in $w$ (beginning with $xb$).
	Thus there exists an occurrence of $xb$ at position $p$, $i<p<j$, otherwise
	we would have another return from $x$ to $y$ in $w$ (starting with $xa$).
	Now $w[p,j)y$ is a return from $x$ to $y$ in $w$, hence it begins with $xb$.
	But this is a contradiction, since there is return to $x$ in $w$ which
	begins with $xb$ and, as we noticed above, there is  also a return to $x$
	in $w$ beginning with $xa$. We conclude that $w$ is a $k$-abelian
	singleton.
\end{proof}

We are now going to describe the structure of $k$-abelian
singletons. For this we need some technical lemmas and
notation.

\begin{lemma} \label{lemma_power}
Let $u$ be a $k$-abelian singleton, and let $x\in\Sigma^{k-1}$ be a
factor of $u$ occurring at least three times. Then $u = u[1,i) v^l
x u[m+k-1..]$, where $v$ is the unique return to $x$, $l \geq 2$
is an integer, and $i$ and $m$ are the positions of the first and
the last occurrences of $x$ in $u$, respectively.
\end{lemma}
\begin{proof} Let $i_0=i, i_1, \dots, i_l=m$, $l\geq 2$, be the
sequence of all positions of $x$ in $u$. Then the words $u[i_j,i_{j+1})$,
$j=0, \ldots, l-1$, are returns to $x$ in $u$. By
\autoref{prop:characterizationReturnwords},
$x$ has exactly one return word in $u$, and this return word is $v$.
It follows that $u[i_j,i_{j+1}) = v$ for all $j=0, \dots, l-1$.
We thus have $u = u[1,i) v^l x u[m+k-1..]$, where 
$v^lx$ contains all occurrences of $x$ in $u$. \end{proof}

We say that a non-empty word $w$ is \emph{$k$-full} if
$|F_{k-1}(w^{\circ})| = |w|$. In other words, $w$ is $k$-full if
$w^{\omega}$ contains $|w|$ distinct words of length $k-1$. Further, we
define a \emph{$k$-full repetition} as a repetition of a $k$-full word,
which contains some factor of length $k-1$ at least 3 times. Clearly,
\begin{itemize}
\item a $k$-full word is primitive,
\item each factor of length $k-1$ in $w^{\omega}$, with $w$ $k$-full, has a unique
return word,
\item the repetition $v^l x$ in \autoref{lemma_power} is $k$-full,
\item in a $k$-abelian singleton, any repetition $r^q$, with
\begin{equation}\label{eq:k-fullRepLength}
	q \geq \tfrac{k-1}{|r|}+2,
\end{equation}
has to be a $k$-full repetition. Indeed, property \eqref{eq:k-fullRepLength}
ensures that the repetition $r^q$ contains at least one factor of length $k-1$ at
least three times.
\end{itemize}

Let $u = a_1a_2\cdots a_n$ and suppose $u[i,m) = r^q$ for some primitive
$r\in\words$ and $q \geq 2$. Then the repetition $r^q$ is called a \emph{run} if
both $a_{i-1}\neq a_{i-1+|r|}$ (or $i=1$), and $a_m\neq a_{m-|r|}$ (or $m-1=n$).
In other words, a run in a word $u$ is a maximal (or non-extendable) repetition
in $u$.

Note that each repetition in a word can be extended to a run (of the same
period) by adding a prefix and a suffix if necessary. So, in the
expression $u = u[1,i) v^l x u[m+k-1..]$ we can now extend the
repetition $v^l x$ to a run:

\begin{corollary}\label{cor_power}
Let $u$ be a $k$-abelian singleton, and let $x\in\Sigma^{k-1}$ be a
factor of $u$ occurring at least three times. Then $u$ is of form
$u = t r^{q} t'$ where $r^{q}$ is a $k$-full run containing $x$,
and any word in $F_{k-1}(r^{\circ})$ occurs only in the run and
nowhere else in $u$.
\end{corollary}

\begin{proof}
First notice that the run $r^{q}$ is an extension of the repetition
$v^l x$ defined by \autoref{lemma_power}, and in particular $q > l$.
Clearly, $r^{q}$ contains all occurrences of $x$, since it contains $v^lx$,
which, in turn, contains all occurrences of $x$. It is not hard to see
that the same is true for all other factors of length $k-1$ of $r^{q}$.
Indeed, each factor of length $k-1$ occurs in it at least twice. So, if
such a factor occurred somewhere else in $u$ (outside the run), then it
would have at least two returns, which contradicts
\autoref{prop:characterizationReturnwords}.
\end{proof}


\begin{example} We illustrate Lemma \ref{lemma_power} and Corollary \ref{cor_power} by the following example: Take $u=0010101010001111$, $k=4$ and $x=101$, then we have
$v=10$, $l=2$
, so that $u = 00\hskip5pt (10)^2 101\hskip5pt 0001111$. Extending
the repetition $(10)^2 101$ to $0(10)^2 1010= (01)^{9/2}$, we get
$u = 0\hskip5pt (01)^{9/2} \hskip5pt001111$. Here $r=01$, $q=9/2$,
$t=0$, and $t'=001111$.
\end{example}

So, in a singleton class, each factor $x$ of length $k-1$ occurring
at least three times, occurs in some run $r^q$ and nowhere else.
Between two different runs $r_1^{q_1}$ and $r_2^{q_2}$ there could
be a word $t$: $r_1^{q_1} [t] r_2^{q_2}$
($t$ might be $\varepsilon$), or they might overlap by a word $t'$
of length at most $k-2$ (because of the condition on returns);
we will denote this as $r_1^{q_1} [t']^{-1} r_2^{q_2}$.
For example, for $k=4$, $u=0110110110010010010$ we write
$u=(011)^{10/3} [10]^{-1} (100)^{11/3}$.



The following proposition gives a structure of the $k$-abelian
singleton classes:

\begin{proposition}\label{prop:singletonstructure}Let $u$ be a $k$-abelian singleton. Then $u$ is of the form
\begin{align}\label{eq:k-abelianSingularFactorization}
    u=t_0 \cdot r_1^{q_1} \cdot [t_1]^{\sigma_1} \cdot r_2^{q_2} \cdot [t_2]^{\sigma_{2}}\cdots r_s^{q_s} \cdot t_s,
\end{align}
where $\sigma_i$ is either $-1$ or $+1$, $t_i\in\words$, $r_i$ is
$k$-full run, $q_i \geq 2 + \frac{k-1}{|r_i|}$ is rational for all $i = 1,\ldots,s$,
and for $i\neq j$ we have $F_{k-1}(r_i^{\circ}) \cap
F_{k-1}(r_j^{\circ}) = \emptyset$. Furthermore, if $\sigma_i = +1$
(resp., $-1$), then any factor of length $k-1$ overlapping $t_i$
(resp., containing $t_i$ as a proper factor) occurs at most twice in $u$.
\end{proposition}

\begin{proof} The proof is in fact the application of \autoref{cor_power} to
all factors occurring at least three times in $u$; these factors give
rise to the $k$-full runs $r_i^{q_i}$. The words $t_i$ come from
joints of runs in a word.
\end{proof}

So, in fact we have two types of factors of length $k-1$:
those which occur in some $r_i^{q_i}$ and can occur more than twice,
and those which overlap $t_i$ for $\sigma_i=+1$ and contain $t_i$
as a proper factor for $\sigma_i=-1$; the latter factors can occur at most twice.

\begin{example}
The word $u = 0^{10} [00]^{(-1)} (0011)^{7/2} [\eps]^{} (01)^{13/2}$ is a
$5$-abelian singleton. The factor $0000$ occurs only in $r_1^{q_1}= 0^{10}$,
the factors $0011$, $0110$, $1100$, $1001$ occur only in
$r_2^{q_2}=(0011)^{7/2}$, and $0101$, $1010$ occur only in
$r_3^{q_3}=(01)^{13/2}$. The factor $0001$ occurs twice, once in the
intersection with $t_1 = [00]^{(-1)}$ and once as an overlap with
$t_2 = [\eps]$. The factor $1000$ occurs once in the overlap with $t_2$.
It is not hard to see that no switching is possible, so the class is
indeed a singleton.
\end{example}

\begin{remark}\label{rem:singletonStructure}
Let $u$ be a word having a representation \eqref{eq:k-abelianSingularFactorization}.
\begin{itemize}
\item If each word of length $k-1$ overlapping some $t_i$ occurs exactly
once in $u$, then $u$ is a $k$-abelian singleton.
\item Some words overlapping $t_i$ with $\sigma_i = +1$
(resp., containing $t_i$ with $\sigma_i = -1$ as a proper factor)
can occur twice. In this case $u$ could be a singleton (if no
switchings are possible) or not (if a switching is possible).
\end{itemize}
\end{remark}

\subsection{The type of a singleton}
We shall fix some notions which we will use further on. Given a singleton
with a representation \eqref{eq:k-abelianSingularFactorization}, we say
that the tuple
\begin{equation}\label{eq:type}
	(\{r_i\}_{i=1}^{s}, \{\langle q_i\rangle\}_{i=1}^s, \{t_i\}_{i=0}^{s}, \{\sigma_i\}_{i=1}^{s-1})
\end{equation}
defines the \emph{type} of the singleton (here $\langle q\rangle $ denotes the
fractional part of a rational number $q$: $\langle q \rangle = q - \lfloor q \rfloor$).
When the type of the singleton is defined, only the integer parts $\lfloor q_i \rfloor$ of
the powers $q_i$ may change. Note here that if one choice of the numbers
$\lfloor q_i \rfloor$ defines a singleton, then so will all other choices, as long as
\eqref{eq:k-fullRepLength} is satisfied for all the runs.

The following lemma says that given $k$ and $\Sigma$, the
number of types of $k$-abelian singletons of length $n$ is bounded
by a constant which does not depend on $n$ (but depends, of course, on $k$):

\begin{lemma}\label{lem:boundedTypes}
Given $k$ and $\Sigma$, the number of types of $k$-abelian
singletons of length $n$ is $\Theta(1)$.
\end{lemma}
\begin{proof}
First notice that the lengths of $r_i$ are bounded. In fact, the sum of
the lengths of $r_i$ is bounded by $|\Sigma|^{k-1}$. Indeed, from
the definition of a $k$-full run we have that
$F_{k-1}(r_i^{\circ})=|r_i|$ and from
\autoref{prop:singletonstructure} we obtain that
$F_{k-1}(r_i^{\circ}) \cap
F_{k-1}(r_j^{\circ}) = \emptyset$.
Therefore, the sum of lengths of $r_i$ is bounded by the
total number of words of length $k-1$ on the alphabet $\Sigma$,
i.e., by $|\Sigma|^{k-1}$. It follows that the number of fractional
parts $\{q_i\}$ is bounded.
 Also, the length of $t_i$ is bounded
(e.g., by 2$|\Sigma|^{k-1}$). Indeed, each $t_i$ can contain, as a
factor, each word of length $k-1$ at most twice, so its length is
at most twice the number of all words of length $k-1$,
i.e., $2|\Sigma|^{k-1}$.

Now, since the lengths of the words $r_i$ and $t_i$ are bounded
and the number of the fractional parts $\{q_i\}$ is bounded, we
conclude that the numbers of all the elements defining the type
of the class are bounded by a constant which does not depend on
$n$. Hence the number of types of $k$-abelian singletons is $\Theta(1)$.
\end{proof}

\section{On the number of singleton classes}\label{sec:numberOfSingletons}
The main goal of this section is to prove the following theorem:

\begin{theorem}\label{th:singletons}
The number of $k$-abelian singleton classes of length $n$ over an
$m$-ary alphabet is of order $\mathcal O (n^{N_m(k-1)-1})$, where
\begin{equation}\label{eq:numberOfNecklaces}
N_m(l)=\tfrac{1}{l}\sum_{d\mid l}\varphi(d)m^{l/d}
\end{equation}
is the number of necklaces of length $l$ over an $m$-ary alphabet and
$\varphi$ is Euler's totient function.
\end{theorem}
The sequence $(N_2(l))_{l=0}^{\infty}$ is sequence A000031 in Sloane's
encyclopedia of integer sequences. The first few values of the sequence
are
\begin{equation*}
1, 2, 3, 4, 6, 8, 14, 20, 36, 60, 108, 188, 352, 632, 1182, 2192, 4116, 7712, 14602,\ldots.
\end{equation*}
See also the sequences A001867--A001869 for alphabets of size $3$--$5$.

We conjecture that this upper bound is tight, i.e., in fact the
number of $k$-abelian singleton classes is of order
$\Theta (n^{N_m(k-1)-1})$ (see \autoref{conj:SingularClasses}).

\subsection{A first upper bound for the number of singletons}
We shall first show that the number of singletons of length $n$ is
bounded by a polynomial in $n$ whose exponent is connected to
representation \eqref{eq:k-abelianSingularFactorization}.

Consider now the set of singletons of length $n$ defined by the type of
form \eqref{eq:type}. The size of this set is equal to the number of
integer solutions $(y_1, \ldots, y_s)$ of the equation
\begin{equation}
    \sum_{i=1}^s|r_i| (y_i + \langle q_i \rangle) + \sum_{i=0}^{s}\sigma_i |t_i| = n.
\end{equation}
Here, for each $i = 1,\ldots,s$, we have the restriction
\begin{equation}\label{eq:variableRestriction}
	y_i\geq \tfrac{k-1}{|r_i|} + 2 - \langle q_i\rangle
\end{equation}
so that the run $r_i^{y_i + \langle q_i\rangle }$ is indeed a $k$-full run. Of
course, the equation might not have any solutions (e.g., for parity reasons). In any
case, letting $n$ grow, the number of solutions $(y_1, \dots, y_s)$ is of order
$\mathcal O \left( n^{s-1} \right)$.

\begin{proposition}\label{prop:smax}
The number of $k$-abelian singletons of length $n$ is of order $\Theta(n^{s_{\max}-1})$,
where $s_{\max}$ is the maximal $s$ among the representations
\eqref{eq:k-abelianSingularFactorization} which correspond to singletons.
\end{proposition}
\begin{proof}
Let the type of a $k$-abelian singleton $u$ be defined by the tuple
\begin{equation*}
(\{r_i\}_{i=1}^{s}, \{\langle q_i\rangle\}_{i=1}^s, \{t_i\}_{i=0}^{s}, \{\sigma_i\}_{i=1}^{s-1}).
\end{equation*}

Consider then the set of singletons of the same type as $u$ and which have length
at least $n$ but less than $n+|r_s|$ (i.e., we allow the $y_i=\lfloor q_i \rfloor$s
to vary). The size of this set is equal to the number of integer solutions
$(y_1, \dots, y_s)$ to
\begin{equation*}
    n\leq \sum_{i=1}^s|r_i| (y_i + \langle q_i \rangle)
        + \sum_{i=0}^{s}\sigma_i |t_i| < n+|r_s|,
\end{equation*}
where each $y_i$, $i=1,\ldots,s$, satisfies \eqref{eq:variableRestriction}. Here we
add the range $|r_s|$ to the length $n$ in order to ensure the existence of solutions
for $n$ large enough.

Each such solution corresponds to a $k$-abelian singleton of length at least $n$ but
less than $n+|r_s|$. By deleting letters from the end, we obtain a singleton (possibly
of a different type, since we modify the end of the word) of length $n$. It is
straightforward to check that two distinct solutions correspond to two distinct
singletons of length $n$. Letting $n$ grow, the number of such solutions is of order
$\Theta(n^{s-1})$, and they correspond to a collection of distinct $k$-abelian
singletons of length $n$.

By choosing $u$ as a singleton with maximal $s$ in its representation, we obtain the
lower bound $\Omega(n^{s_{\max}-1})$ in the claim. Furthermore, by
\autoref{lem:boundedTypes}, the number of distinct types is of order $\Theta(1)$. By
summing over all types we get the upper bound $\mathcal O(n^{s_{\max}-1})$. The claim
follows.
\end{proof}

We shall now proceed to obtain upper bounds for $s_{\max}$ in the above proposition.
For this we observe a connection between $k$-abelian singletons and \emph{cycle
decompositions} of the de Bruijn graph of degree $k-1$.

\subsection{Singletons as cycle semi-decompositions of de Bruijn graphs}\label{subsec:singletonAsDecomposition}
Consider now a $k$-abelian singleton $u$ with representation
\eqref{eq:k-abelianSingularFactorization} as a path in the de Bruijn graph $dB(k-1)$.
Each $k$-full run $r^q$ corresponds to a cycle in de Bruijn graph. The path induced
by $u$ in $dB(k-1)$ can be seen to enter and leave these cycles, never to return
again. A $k$-abelian singleton can thus be seen as a decomposition of $dB(k-1)$ into
vertex disjoint cycles (corresponding to the runs $r_i^{q_i}$) which are connected
by certain paths (defined by the words $t_i$). As we are interested in maximizing
$s$ in representation \eqref{eq:k-abelianSingularFactorization}, the above translates
to finding a decomposition of $dB(k-1)$ into the largest number of cycles which are
connected by paths. We shall now make the above discussion rigorous.

Let $G = (V,E)$ be a graph and let $C = \{C_1,\ldots, C_m\}$ be a set of
vertex-disjoint cycles of $G$. Let $V_i$ be the set of vertices in $C_i$, and let
$V_{\otimes} = V \setminus \bigcup_{i=1}^m V_i$. The set consisting of the partitions
$V_i$, $i=1,\ldots,m$ and $\{v\},v\in V_{\otimes}$, is called a \emph{cycle
semi-decomposition} of $G$, denoted by $V/C$.

\begin{definition}
Let $G=(V,E)$ and let $V/C$ be a cycle semi-decomposition. We define the quotient
graph $G/C = (V/C,E')$ with respect to $C$ as follows. For $X,Y\in V/C$, $X \neq Y$,
we have $(X,Y) \in E'$ if and only if there exist $x\in X$ and $y\in Y$ such that
$(x,y)\in E$.
\end{definition}

Let $u$ be a $k$-abelian singleton. Consider the graph $G_u$ obtained from
$G_{\Psi_k(u)}$ by removing multiplicities of edges. Let $C_u = \{C_1,\ldots,C_s\}$ be
the set of vertex disjoint cycles; the sets of vertices which the path induced by $u$
visits at least three times. The graph $G_u/C_u$ then contains a path which traverses
through each $V_i$ once and through each $\{v\}$, $v\in V_{\otimes}$,
at least once and at most twice.

\begin{example}
Consider the graph $G_u$ induced by the $4$-abelian singleton
$u = 2 (01)^3 [2] (0110)^{15/4} [011]^{(-1)} 1^6$. The vertex-disjoint cycles
corresponding to $u$ are $V_1 = \{010,101\}$, $V_2 = \{011,110,100,001\}$ and
$V_3 = \{111\}$. The set of factors occurring at most twice in $u$ is
$V_{\otimes} = \{201,012,120\}$. The quotient graph $G_u/C_u$ is displayed in
\autoref{fig:quotientGraph}.
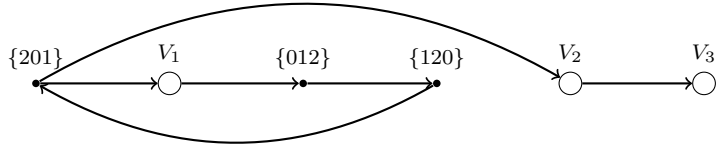
\begin{figure}
    \centering
    \begin{tikzpicture}[node distance=50pt, cycle node/.style={inner sep=3pt,circle,draw},
            single node/.style={inner sep=1pt, circle,fill}, every node/.style={rectangle}]

\node[single node] [label={\footnotesize{$\{201\}$}}] (A) {};
\node[cycle node] (B) [right of=A,label={\footnotesize{$V_1$}}] {};
\node[single node] (V1) [right of=B,label={\footnotesize{$\{012\}$}}] {};
\node[single node] (V2) [right of=V1,label={\footnotesize{$\{120\}$}}] {};
\node[cycle node] (C) [right of=V2,label={\footnotesize{$V_2$}}] {};
\node[cycle node] (D) [right of=C,label={\footnotesize{$V_3$}}] {};

\draw (A) edge[->,thick] (B);
\draw (B) edge[->,thick] (V1);
\draw (V1) edge[->,thick] (V2);
\draw (V2) edge[->,thick,out=210,in=330] (A);
\draw (A) edge[->,thick,out=30,in=150] (C);
\draw (C) edge[->,thick] (D);

\end{tikzpicture}
\caption{The quotient graph $G_u/C_u$ induced by $u =
2 (01)^3 [2] (0110)^{15/4} [011]^{-1} 1^6$.} \label{fig:quotientGraph}
\end{figure}
\end{example}

By \autoref{rem:singletonStructure}, we obtain the following.

\begin{proposition}\label{prop:smaxcycles}
Let $G$ be the de Bruijn graph of order $k-1$. Then $s_{\max}$ (as in \autoref{prop:smax})
is bounded by the largest $s$ such that $C$ is a set of cardinality $s$ of vertex-disjoint
cycles and $G/C$ contains a path which traverses each $V_i$ precisely once and each
$\{v\}$, $v\in V_{\otimes}$, at most twice.
\end{proposition}

\begin{remark}
Note that $s_{\max}$ is bounded below by the largest $s$ such that $C$ is a set of size
$s$ of vertex-disjoint cycles and $G/C$ contains a path which traverses through each $V_i$
precisely once and through each $\{v\}$, $v\in V_{\otimes}$, at most once.
\end{remark}

The following theorem, originally known as Lempel's conjecture, was proved by J. Mykkeltveit
in \cite{Myk72}.

\begin{theorem}\label{thm:MykkeltveitLempel}
The minimum number of vertices which, if removed from $dB_{\Sigma}(n)$, will leave a
graph with no cycles, is $N_{|\Sigma|}(n)$ (defined by \eqref{eq:numberOfNecklaces}).
\end{theorem}

It follows that a set of vertex-disjoint cycles of $dB_{\Sigma}(n)$ can contain at most
$N_{|\Sigma|}(n)$ cycles. By the above theorem and \autoref{prop:smaxcycles}, we
immediately obtain \autoref{th:singletons}.

\begin{definition}
A cycle semi-decomposition $V/C$ of $dB_{\Sigma}(n)$ is called \emph{maximal} if
$C$ contains $N_{|\Sigma|}(n)$ cycles.
\end{definition}

Note that maximal cycle (semi-)decompositions exist for any $n\in N$: take the cycles induced
by necklaces. However, the above theorems do not give $\Omega(n^{N_{|\Sigma|}(k-1)-1})$ for
the number of $k$-abelian singletons, since it gives only the maximal number of cycles in
the cycle decomposition of the de Bruijn graph: we would also need a path in the quotient
graph containing those cycles, i.e., we do not know whether the upper bound for $s_{\max}$
is achievable.

We now show that a maximal cycle semi-decomposition is actually a cycle decomposition,
that is, each vertex occurs in one of the cycles and $V_{\otimes}$ is empty.

\begin{proposition}\label{prop:maxVEmpty}
For any maximal cycle semi-decomposition $V/C$ of $dB_{\Sigma}(n)$,
each vertex occurs in one of the cycles of $C$.
\end{proposition}
\begin{proof}
We first recall the following: Let $G$ be an Eulerian graph and $\tilde C$ a set of
edge-disjoint cycles of $G$. Then there exists a decomposition $\tilde D$ of $G$ into
edge-disjoint cycles such that $\tilde C \subseteq \tilde D$.
Indeed, since $G$ is Eulerian, each vertex $v$ has the property $d_G^+(v) = d_G^-(v)$. If $G'$
is the graph obtained from $G$ by removing edges occurring in $\tilde C$, then each vertex
$v\in G'$ has the property $d_{G'}^+(v) = d_{G'}^-(v)$. It follows from Veblen's theorem
(\cite{Veb12}) for directed graphs (see, e.g., \cite{BMu08}, exercise 2.4.2) that there exists a
decomposition $\tilde E$ of $G'$ into edge-disjoint cycles. Now $\tilde E \cup \tilde C = \tilde D$
is a decomposition of $G$ into edge-disjoint cycles satisfying the claim.

Let then $C$ be a set of $N(n)$ vertex-disjoint cycles in $dB(n)$. The vertices of $dB(n)$
correspond to the edges of $dB(n-1)$, so that $C$ can be seen as a set $\tilde{C}$
of \emph{edge}-disjoint cycles of $dB(n-1)$, an Eulerian graph.

Suppose there is a vertex $v$ in $dB(n)$ not included in any of the cycles of $C$. Then
$v$ corresponds to an edge $e$ in $dB(n-1)$ which does not occur in $\tilde{C}$. By the
above, $\tilde C$ can be extended to a decomposition of $dB(n-1)$ into edge-disjoint cycles
$\tilde D$ such that $\tilde C \subset \tilde D$. But now $\tilde D$ can
be seen as a set of vertex-disjoint cycles of $dB(n)$ with more than $N(n)$ cycles, a
contradiction.
\end{proof}

\section{Maximal cycle decompositions and Gray codes for necklaces}
In this section we focus on maximal cycle decompositions. We are interested in finding a
maximal cycle decomposition $\Sigma^n/C$, such that $G=dB(n)/C$ contains a \emph{Hamiltonian path},
i.e., a path which visits each vertex precisely once. We note that, for a maximal cycle
decomposition of $dB(n)$, the quotient graph can be seen as undirected. To see this, let $(X,Y)$
be an edge of $G$, that is, there exist $a,b\in \Sigma$, $u\in \Sigma^{n-1}$ such that $au\in X$
and $ub\in Y$ whence $(au,ub)$ is an edge in $dB(n)$. By \autoref{prop:maxVEmpty}, $X$ and $Y$ are
cycles, so that there exist $c,d\in \Sigma$ such that $(au,uc)\in X$ and $(du,ub)\in Y$, whence
$(du,uc)$ is an edge in $dB(n)$. By definition, $(Y,X)\in G$ as well.

\subsection{On necklace graphs and Gray codes for necklaces}
We shall first consider the cycle decomposition of the de Bruijn graph given by the cycles
induced by necklaces. Note that the length of such a cycle divides the order of the de
Bruijn graph. Compared to the discussion in the beginning of \autoref{subsec:singletonAsDecomposition},
we have a special case where the lengths of the roots of the $k$-full runs divide $k-1$.
We begin with a definition.
\begin{definition}
Let $C_{\Sigma}(n)$ be the set of cycles induced by necklaces of length $n$.
The quotient graph $NG_{\Sigma}(n) = dB_{\Sigma}(n)/C_{\Sigma}(n)$ is called
the \emph{necklace graph} of order $n$.
\end{definition}

The following example shows that $NG(n)$ does not always contain a Hamiltonian path, so
that necklace graphs will not provide what we need.

\begin{example}
The binary necklace graphs of order $4$ and $5$ are illustrated in \autoref{fig:necklacegraphs}.
A longest path in $NG(4)$ contains 5 vertices out of a total of 6. On the other hand, one can
easily find a Hamiltonian path in $NG(5)$.
\begin{figure}
  \centering
  \begin{tikzpicture}[node distance = 60pt, every node/.style={ellipse,inner sep=.5pt,draw}]

    \def\x{5pt}
    \def\y{25pt}
    \node (1) at (0,0) {\footnotesize{$0000$}};
    \node (2) [right of = 1] {\footnotesize{$0001$}};
    \node (3) [above right =\x and \y of 2] {\footnotesize{$0011$}};
    \node (4) [below right = \x and \y of 2] {\footnotesize{$0101$}};
    \node (5) [below right = \x and \y of 3] {\footnotesize{$0111$}};
    \node (6) [right of = 5] {\footnotesize{$1111$}};

    \foreach \s/\t in {1/2,2/3,2/4,3/5,4/5,5/6}{
        \draw (\s) edge[thick] (\t);
    }
  \end{tikzpicture}
\vskip20pt
  \begin{tikzpicture}[node distance=60pt,every node/.style={ellipse,inner sep=.5pt,draw}]
        \def\x{5pt} 
        \def\y{25pt} 

        \node (1) at (0,0) {\footnotesize{$00000$}};
    \node (2) [right of=1] {\footnotesize{$00001$}};
    \node (3) [above right=\x and \y of 2] {\footnotesize{$00011$}};
        \node (4) [below right=\x and \y of 2] {\footnotesize{$00101$}};
        \node (5) [right of=3] {\footnotesize{$00111$}};
        \node (6) [right of=4] {\footnotesize{$01011$}};
        \node (7) [below right=\x and \y of 5] {\footnotesize{$01111$}};
        \node (8) [right of=7] {\footnotesize{$11111$}};

        \foreach \s/\t in {1/2,2/3,2/4,3/5,3/6,4/5,4/6,5/7,6/7,7/8}{
            \draw (\s) edge[thick] (\t);
        }
  \end{tikzpicture}
  \caption{The binary necklace graphs $NG(4)$ and $NG(5)$. A vertex
	represents the necklace induced by its label.}
  \label{fig:necklacegraphs}
  \end{figure}
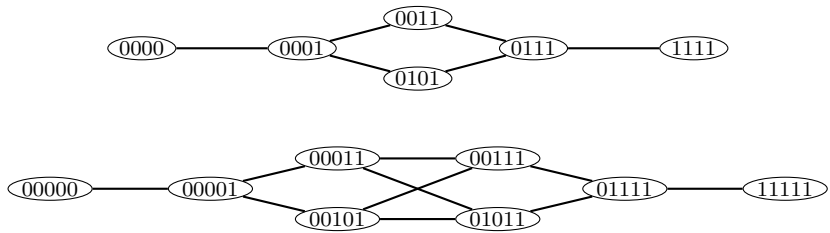 
\end{example}

The problem of finding a Hamiltonian path in necklace graphs has been extensively studied in
terms of \emph{Gray codes} (\cite{DeDr07,Sav97,SRu99,WVa06,ZKC08}). A \emph{Gray code for necklaces}
of length $n$ is defined as a sequence of all necklaces of length $n$ such that two
consecutive necklaces have representatives which differ in one bit. One can easily see that
Gray codes for necklaces correspond to Hamiltonian paths in necklace graphs.
The following conjecture was stated almost 20 years ago. To the best of our knowledge it
remains unsolved to this day.

\begin{conjecture}[\cite{Sav97}, section 7]\label{con:GrayCode}
Let $n\in\N$ be odd and let $\Sigma$ be a binary alphabet. Then there exists a Gray
code for necklaces of length $n$. In other words, $NG(n)$ contains a Hamiltonian path.
\end{conjecture}

The above has previously been verified for necklaces up to length $9$ in \cite{ZKC08}.
We have computationally verified the conjecture till $n=15$. For completeness, we give
Gray codes for binary words of lengths 5--15 in \autoref{tab:GrayCodeOdd}. There we
represent a Gray code as a word $a_1\ldots a_{N(n)-1}$ over the \emph{hexadecimals} to be
read as follows. The first necklace in the ordering is $0^n$. The $(i+1)$st necklace
is then obtained by complementing the $a_i$th letter ($0\leftrightarrow 1$) of the
lexicographically least representative of the $i$th necklace. For example, the coding
$1114111$ corresponds to the following ordering of necklaces of length $5$:
\begin{equation*}00000,00001,00011,00111,00101,01011,01111,11111.\end{equation*}

\begin{table}[p]
\scriptsize
\begin{tabular}{|l|p{10.6cm}|}
\hline 
$n$ & Gray code\\
\hline
\hline 
$5$ & \texttt{1114111}\\
\hline 
$7$ & \texttt{1116165614521341111}\\
\hline 
$9$ & \texttt{11181878167876781576876567861878678185951575415813754113211}\\
\hline 
$11$ & {\texttt{111a1a9a189a989a179a989a9798a9891679a89a978a98a6789a89789a871579a978a98a67a89a8a78
98b6167a9a879697a89a97a5856896378a2789a87269a78a989678798489278979a62a8a7292474525
4527118584a82a346181811}
} \\
\hline 
$13$ & {\texttt{111c1cbc1abcbabc19bcbabcb9bacbab189cbcabcb9abacba89bcbabcb9bcaba1789abcab9acbcac98
abcbabc78abcab9acbcac98abcbab97acbcac7abcba16789cbc9abcbab89cabcac9abcba789acbac9a
cbc89cbacbc89b65189abcb9abcac89cabcac9879acabc8abcab96acba6915677ca1579bacba9a8cbc
ab1ca78b9cabcb8797b8cba89cbacb9b768c9bc9acb89bcabcb9856ab8abcb9acbac89cbca97c7ba75
68cb89acabc9789cba9babc98acabda7acb188acabc78cbc8ca967cb7ab76a8bacba8a57c5a8bacba6
14ccb2bc89aca9aba789bc9acb8b67963acab567bc9acbab789b9abca98ca9ba678ba89abc7ac859c9
7897ca98ba4a9ca2c67cb98747a158ca9bc89ca7a89b8d6167bc74a674ca47c45956c89aca6a42c11b
18bc92c8ba86ca657b52865628ba392952178a712ca139c78197a4411}
} \\
\hline 
$15$ & {\texttt{111e1ede1cdedcde1bdedcdedbdcedcd1abedecdedbcdcedcabdedcdedbdecdc19abcedecedbcedece
abcdedcde9acedecedbcedeced17dedcd9bcedecebcdedc189abcdecdbcedecebacdedcde9acedeced
bcedeced89bcdecdbcedecebacdedcdebc8edeced9bcedecebcdedc89cdc9abcedecedacedeced16ed
ece178abedebcdedcdabcedeced9acdecdbcedecebacdedcdbc9edece9cdca89bcdecdbcedecebacde
dcde9acdecdbcedecebacdedcdb9cedece9cded8c8789abdedcbdedabdedcded9acedceb9edacdece9
acedecbacdedbdedcf9cded18eb89abedcecbcdedcd9bcdecdbcdecba9cdedcde167bdcedcbdecede9
edecd1679edbecdedbdedca8bde3cdedbdcedcd9becdecea8cdcedcbdcedca8dedced89bcdecdbce9b
decded9bec9dcda789bedbcedecabcdcedcbded9badcedcbdece9ae9bcbedbcdbd98abdedcbdedabed
cedebadeaced9f7c9edece179aedecdecded8bab8cdcedcbadcedc678becdedbdcedcda9becdecebcd
edc9bcedce89aedcbdedcbacedcdbadeacedec9abcecdebcecda67d96f819bcdece8bcded8a789dedc
bacdcedcbdedcdaedad9abedcecbedceaebdecdedbcec86dce67ec79abedebcdedcd9bedcecbedca5d
ec19bcdcedcbded971579adedbcedceabecdecbcedcda9bcecdebcdecdabecdecbcedcd789acdcedca
bedcede9acecdebcdecda9ecd9bdcedcbc98abedbcdeabedcedeba9867bceedcecdedbcedce9aecd9b
dcedcbcde8c86789decb9adebcdecda9ed9bcdce89eadebcedceaecbdeceba9bcbdcbedbca8b8edcde
b879adedcabdcedcb9aebcdedbdeca9ce9becdece9b9c789ec7abebcdceabecdecebcdc97875f467ab
ebcdceabeced2edeb9decded15bcbdcbedbc9a9bedcedeb7569bcec9adedcbdcedc9aedc98cdecd8be
dcebed789dbecdedbadaeda9bdedcbdceabedcecbc98aedcedebacecdebcdc98a8768bcedccdedbc86
7edcbedbdabedcecb98acedceabedabacbec9adedaceb8edcd79a9bdedbdecda68adcea968bedcecbc
edc87c9adecd5a7a59ecdeceabcecde9f46be68ecded8adedbcedc7abebd67bdeb81689de71edaedcb
e1ebeabcecdbcdc89acdcabdcecbacaed9adecbedbd78bdbcdecba9ecbdcb9dceb9ce9cdc87bdcecbc
e79aecedabdcebce968dc1ebedbcdeabecd5679dce2bcedcd98cd38569bedbcde9aedcbdebea98acab
cdba97bdebcedadbcebc79ca9bed6dabdcbd89dacdabdc9caeb8ab89ba7ecbcdc96bdcbd8965958ade
cbebd87cdea879c8b6986ecbc7897bcbe745eddcbddeab9aebcec8de6ed4c92ecbedbdaeb676459dab
dc9ca78bda9dbab45dcbdd9bcb635f4de51676b2dab76ba6997ec89a9ce7ec9ea7978bc8ce45e6cbce
6746bec1e78aecbacbd9b74319db8de9a8cec9acecb986511398b7959edbda796c8cd8a867a8dc1aec
eb1d425975ae7ec425527e11d1dbadcac1642aca8acd8abd89bd96e729a8595dbeb7265bd18a71dcdc
3c4562b8676d7b74169b3115278631cca414f52ed1db62ad416187cb111}
} \\
\hline
\end{tabular}
\caption{Gray code for binary necklaces of odd length $n\leq 15$.}
\label{tab:GrayCodeOdd}
\end{table}

On the other hand, binary necklace graphs are bipartite. When $n\geq 4$ is even,
the difference of the partitions is greater than 1 so the graph cannot contain
a Hamiltonian path. In fact, it is not hard to calculate an upper bound:

\begin{proposition}
For the binary alphabet and $n$ even, the number of vertices in a longest path in $NG(n)$
is at most
\begin{equation}\label{eq:longestEvenPathLength}
BPL(n) = \frac{1}{n}\sum_{\substack{d\mid n\\ 2\nmid d}}\varphi(d)2^{n/d}+1.
\end{equation}
\end{proposition}
The first few terms of $(BPL(2n))_{n=1}^{\infty}$ are
$$3,\ 5,\ 13,\ 33,\  105,\  345,\  1173,\  4097,\  14573,\  52433,\  190653,\  699073,\ldots.$$
The sequence $(BPL(n))_{n=1}^{\infty}$ equals $(a(n)+1)_{n=1}^{\infty}$ where $(a(n))_{n=1}^{\infty}$
is sequence A063776 in Sloane's encyclopedia of integer sequences.
\begin{proof}
For ease of notation, we shall denote by $(i,j)$ the \emph{greatest common divisor of $i$ and $j$}.
Let $A$ be the set of necklaces containing an even number of $1$'s and $B$ the set of necklaces
containing an odd number of $1$'s; $NG(n)$ is then bipartite with respect to the partition into $A$
and $B$. We have that the number $N(n,l)$ of necklaces of length $n$ containing precisely $l$ $1$s
equals $\tfrac{1}{n}\sum_{d\mid (l,n)}\varphi(d)\binom{n/d}{l/d}$
(see, e.g., \cite{SRu99}). We thus have
	\begin{align*}
	 n|A|	&= \sum_{l=0}^{n/2}nN(n,2l) = \sum_{l=0}^{n/2} \sum_{d\mid (2l,n)}\varphi(d)\binom{n/d}{2l/d}.
	\end{align*}
	We shall count the above in a different order. Let us first consider a fixed divisor $d$
	of $n$ such that $2\mid d$. In the above sum, we count $\varphi(d)\binom{n/d}{2l/d}$ for
	each $0\leq 2l\leq n$ such that $d\mid 2l$, that is, $2l = dl'$ for some $l'$ (since
	$2\mid d$). We thus count $\varphi(d)\binom{n/d}{l'}$ for each $0\leq l' \leq \tfrac{n}{d}$.
	
	Consider then a fixed divisor $d$ of $n$ such that $2\nmid d$. Similar to the above, we count
	$\varphi(d)\binom{n/d}{2l/d}$ for each $0\leq 2l \leq n$ such that $d\mid 2l$, that is,
	$2l = 2dl'$ for some $l'$ (since $2\nmid d$). We thus count $\varphi(d)\binom{n/d}{2l'}$ for
	each $0\leq l' \leq \tfrac{1}{2}\tfrac{n}{d}$.
	
	Combining the above calculations we obtain
	\begin{align*}
		n|A|&= \sum_{\substack{d \mid n\\ 2\mid d}}\varphi(d) \sum_{l'=0}^{n/d}\binom{n/d}{l'}
						+ \sum_{\substack{d \mid n\\ 2\nmid d}} \varphi(d) \sum_{l'=0}^{\tfrac{1}{2}n/d}\binom{n/d}{2l'}\\
				&= \sum_{\substack{d\mid n\\ 2\mid d}} \varphi(d)2^{n/d}
						+ \sum_{\substack{d\mid n\\ 2\nmid d}}\varphi(d)2^{n/d-1}
					= \tfrac{n}{2} N_2(n) + \tfrac{1}{2}\sum_{\substack{d\mid n\\ 2\mid d}}\varphi(d)2^{n/d},
	\end{align*}
	so that $|B| = \tfrac{1}{2} N_2(n) -\tfrac{1}{2n}\sum_{\substack{d\mid n\\ 2\mid d}}\varphi(d)2^{n/d}
	=\tfrac{1}{2n}\sum_{\substack{d\mid n\\2\nmid d}}\varphi(d)2^{n/d}<|A|$.
	A longest path in $NG(n)$ can thus contain at most $|B|+1$ vertices from $A$ and $|B|$ vertices
	from $B$, that is, $2|B|+1 =\tfrac{1}{n}\sum_{\substack{d\mid n\\ 2\nmid d}}\varphi(d)2^{n/d} + 1 = BPL(n)$
	vertices in total.
\end{proof}

We conjecture that this bound is actually achievable:
\begin{conjecture}\label{con:evenLongestPath}
For even $n$, the length of a longest path in the (bipartite)
binary necklace graph $NG(n)$ is equal to $BPL(n)$ (defined by
\eqref{eq:longestEvenPathLength}).
\end{conjecture}

We have computationally verified the conjecture to be true for $n\leq 8$. See
\autoref{fig:necklacegraphs} for $n=4$ and \autoref{tab:codeMaximalEven} for
$n=6,8$, where the coding is defined as that of the Gray codes in
\autoref{tab:GrayCodeOdd}.

\begin{table}
\centering
\begin{tabular}{|l|l|}
\hline 
$n$ & code\\
\hline
\hline 
$6$ & \texttt{111521651511}\\
\hline 
$8$ & \texttt{11171767156725671472674521615611}\\
\hline 
\end{tabular}
\caption{Path in the binary necklace graph $NG(n)$ of length $BPL(n)$.}
\label{tab:codeMaximalEven}
\end{table}

\subsection{On other maximal cycle decompositions of the de Bruijn graph}
We shall now turn to maximal cycle decompositions not induced by necklaces. In other words, the
length of a cycle in such a decomposition need not divide the order of the de Bruijn graph. We
give some examples of such decompositions which induce quotient graphs containing Hamiltonian paths.

\begin{example} The $5$-abelian singleton
\begin{equation*}
u = 0^{i_1}[000]^{-1}(00011)^{i_2+3/5}[001]^{-1}(001)^{i_3+1/3}(01)^{i_4}[101]^{-1}(0111)^{i_5}[111]^{-1}1^{i_6}
\end{equation*}
corresponds to a cycle decomposition (with $V_{\otimes}$ empty) of $dB(4)$, the resulting
quotient graph containing a Hamiltonian path. Moreover, the cycle decomposition contains
$N_{2}(4)=6$ cycles, which is maximal possible by \autoref{thm:MykkeltveitLempel}. Note here
that the second and third cycles have lengths which do not divide $4$.
\end{example}

Similarly, we obtain the following $N_2(6) = 14$ vertex-disjoint cycles in $dB(6)$ which
induce a quotient graph containing a Hamiltonian path:
\begin{equation*}
0, 0^5101, 0^31, 0^411,  0^31^3,  001^4, 001011,
001,001101,011,01,0101^3,01^5,1.
\end{equation*}
(The order the cycles are listed in gives such a path.) Note that the second and third
cycle have lengths which do not divide $6$.

For $n=8$ we computed the following set of $N_2(8) = 36$ vertex-disjoint cycles of $dB(8)$, the cycles
listed in an order yielding a Hamiltonian path in the quotient graph:
\begin{align*}
    & 0, 0^71, 0^611, 0^51^3, 0^5101, 0^41101, 0^41001, 0^41011,
    0^41^4,      0^31^5, 0^31^301, 0^311001, \\
    & 0^31, 0^310101, 01,  010101^5, 00110101, 0^310011,0^311011, 011 , 0101101^4011, \\
    & 00101101, 00100101, 001^3001, 001^401, 0^21^6, 001101^3, 0011, 001^3011, 00101011, \\& 00101^4, 0^3101^3, 0101^3,
    01^3, 01^7, 1.
\end{align*}

These observations lead us to state the following conjecture, which is now verified for
all odd $n\leq 15$ and all even $n\leq 8$ over the binary alphabet.
\begin{conjecture}\label{conj:SingularClasses}
For every $n\in\N$ and alphabet $\Sigma$, there exists a maximal cycle decomposition
of $dB_{\Sigma}(n)$ so that the quotient graph contains a Hamiltonian path.
\end{conjecture}

An equivalent formulation, due to \autoref{prop:smaxcycles}, is:

\begin{conjecture}\label{con:singletonNumberSharp}
   For any $k,m\geq 1$, the number of $k$-abelian singleton classes of length $n$ over an $m$-ary alphabet is of order $\Theta(n^{N_m(k-1)-1})$.
\end{conjecture}

\section{Conclusions}\label{sec:conclusions}
In this paper we were interested in cardinalities 
of $k$-abelian equivalence classes. We were also interested in the
structure of singleton classes. By showing a new equivalent
definition of $k$-abelian equivalence based on rewriting, we obtained a
partial description of the structure of $k$-abelian singletons. Further,
using cycle decompositions of de Bruijn graph, we provided an upper
bound for the number of singleton classes (\autoref{th:singletons}).
We conjecture that this bound is asymptotically sharp and propose
two related conjectures concerning necklace (de Bruijn) graphs
(Conjectures \ref{con:singletonNumberSharp}, \ref{con:GrayCode},
\ref{con:evenLongestPath}). To conclude, we suggest the following
open problem:
\begin{openproblem}
For which functions $f:\N \to \N$ there exists a sequence of
words $(w^{(n)})_{n=1}^{\infty}$, $|w^{(n)}| = n$, such that
$|[w^{(n)}]_k| = \Theta(f(n))$?
\end{openproblem}
In \autoref{claim:example} we obtain a positive answer for any
polynomial $f$ and for any $f$ satisfying $f(n) \leq n-2$ for
all $n\in\N$.

The formula obtained in \autoref{prop:classSize} gives some
hints of what such functions $f$ can be, but to analyze the
asymptotic cardinality seems to be difficult. Even analyzing the
sequence $f_k(n) = \max\{|[w]_k| \mid |w| = n\}$ is nontrivial.

\section*{Acknowledgments}
The first and fourth authors are supported by the Academy of Finland, grants 257857 and 137991.
The second author is supported by the LABEX MILYON (ANR-10-LABX-0070) of Universit\'e de Lyon,
within the program "Investissements d'Avenir" (ANR-11-IDEX-0007) operated by the
French National Research Agency (ANR).
We would like to thank the anonymous reviewers for valuable comments which significantly
improved the presentation.

\section*{\refname}
\bibliography{bibliography}{}
\bibliographystyle{abbrv}

\end{document}